\documentclass[reqno,twoside]{amsart}

\usepackage{amsfonts}
\usepackage{amsmath,amssymb,amsthm}
\usepackage{mathtools}
\usepackage{etoolbox}
\usepackage{color}
\usepackage[english]{babel}
\usepackage{hyperref}
\usepackage{microtype}
\usepackage{fullpage}
\usepackage{graphicx}
\usepackage{floatrow}
\usepackage[utf8]{inputenc}
\usepackage[T1]{fontenc}
\usepackage{lmodern}
\usepackage{mathrsfs}
\usepackage{microtype}
\DisableLigatures{encoding = *, family = * }

\newcommand{\RR}{\mathbb{R}}
\newcommand{\norm}[2]{\left\|#1\right\|_{#2}}

\newcommand{\Fe}{\mathcal F_\varepsilon}
\newcommand{\Je}{\mathcal J_\varepsilon}
\newcommand{\Jeh}{\widehat{\mathcal J}_\varepsilon}
\newcommand{\ue}{u_\varepsilon}
\newcommand{\pe}{{\bf p}_{\varepsilon,\pi}}

\newcommand{\xe}{{\bf x}_\varepsilon}

\newcommand{\mS}{\mathcal{S}}
\newcommand{\aT}{{\bf a}_\pi}
\newcommand{\bT}{{\bf b}_\pi}

\newtheorem{theorem}{Theorem}[section]
\newtheorem{definition}[theorem]{Definition}

\newtheorem{remark}[theorem]{Remark}
\newtheorem{problem}{Problem}[section]

\numberwithin{equation}{section}

\title{Multilevel Selective Harmonic Modulation by Duality} 

\author{Umberto Biccari\textsuperscript{\,$\ast$}}  
\address{\textsuperscript{$\ast$}\, [1] Chair of Computational Mathematics, Fundaci\'on Deusto, Avenida de las Universidades 24, 48007 Bilbao, Basque Country, Spain} 
\address{[2]\,Facultad de Ingenier\'ia, Universidad de Deusto, Avenida de las Universidades 24, 48007 Bilbao, Basque Country, Spain.}
\email{umberto.biccari@deusto.es, u.biccari@gmail.com}
\thanks{This project has received funding from the European Research Council under the European Union’s Horizon 2020 research and innovation programme (grant agreement 694126-DyCon). The work of U.B. and E.Z. is supported by the Grant PID2020-112617GB-C22 KILEARN of MINECO (Spain), by the Elkartek grant KK-2020/00091 CONVADP of the Basque government and by the Air Force Office of Scientific Research (AFOSR) under Award FA9550-18-1-0242. The work of E.Z. is funded by the Alexander von Humboldt-Professorship program, the European Unions Horizon 2020 research and innovation programme under the Marie Sklodowska-Curie grant agreement 765579-ConFlex and the Transregio 154 Project ‘‘Mathematical Modelling, Simulation and Optimization Using the Example of Gas Networks’’, project C08, of the German DFG}

\author{Enrique Zuazua\textsuperscript{\,$\ddagger$}}
\address{\textsuperscript{$\ddagger$}\, [1] Chair for Dynamics, Control and Numerics, Alexander von Humboldt-Professorship, Department of Data Science, Friedrich-Alexander-Universit\"at Erlangen-N\"urnberg, 91058 Erlangen, Germany.}
\address{[2] Chair of Computational Mathematics, Fundaci\'on Deusto, Avenida de las Universidades 24, 48007 Bilbao, Basque Country, Spain.} 
\address{[3] Departamento de Matem\'aticas, Universidad Aut\'onoma de Madrid, 28049 Madrid, Spain.}
\email{enrique.zuazua@fau.de}

\keywords{Selective Harmonic Modulation, Optimal Control Theory, Duality, Piece-wise affine penalization, Multilevel controls.}

\begin{document}
	
\maketitle 

\begin{abstract} 
We address the Selective Harmonic Modulation (SHM) problem in power electronic engineering, consisting in designing a multilevel staircase control signal with some prescribed frequencies to improve the performances of a converter. In this work, SHM is addressed through an optimal control methodology based on duality, in which the admissible controls are piece-wise constant functions, taking values only in a given finite set. To fulfill this constraint, the cornerstone of our approach is the introduction of a special penalization in the cost functional, in the form of a piece-wise affine approximation of a parabola. In this manner, we build optimal multilevel controls having the desired staircase structure. 
\end{abstract}

\section{Introduction}\label{sec:intro}

In power electronic engineering, Selective Harmonic Modulation (SHM) \cite{Sun1992,Sun1996} is a well-known technique, employed to improve the performances of a converter by increasing its power and, at the same time, reducing its losses. In broad terms, the scope of SHM is to control the phase and amplitude of the harmonics in the converter's output voltage, by generating a control signal with a desired harmonic spectrum through the modulation of specific lower-order Fourier coefficients. 

In practice, this signal consists in a step function with a finite number of switches, taking values only in a given finite set. This function can be fully characterized by:
\begin{itemize}
	\item[1.] The \textit{waveform}, i.e. the sequence of values that the function takes in its domain.
	\item[2.] The \textit{switching angles}, i.e. the sequence of points where the signal switches from one value to following one. 
\end{itemize}

Hence, solving the SHM problem consists in determining the waveform and switching angles in the control signal.

SHM has been extensively studied in the engineering literature in the last decades. Nowadays, many different techniques are available to approach this problem, among which we highlight:
\begin{itemize}
	\item[1.] \textit{Conversion into a finite-dimensional optimization problem}. It consists in fixing a suitable waveform for the desired control signal and finding the optimal location of the switching angles \cite{perez2017}. This approach has the main drawback of requiring an a priori knowledge of the waveform, which may be quite cumbersome in some situations. In fact, even determining the number of switching angles is not straightforward in general.
	\item[2.] \textit{Conversion into a polynomial system}. This approach is based on employing appropriate algebraic transformations to convert the SHM problem into a polynomial system whose solutions' set contains all the possible waveforms for a predetermined number of switching angles (see \cite{yang2016} and the references therein). Hence, also in this case, it is necessary to fix a priori the number of switching angles, which is something that, in general, is preferable to avoid. 
	\item[3.] \textit{Conversion into a constrained optimal control problem}. This is a very novel approach, which has been firstly proposed in our recent contribution \cite{biccari2021A}. It consists in identifying the Fourier coefficients of the signal with the terminal state of a controlled dynamical system, where the control is actually the signal, which is computed by minimizing a suitable cost functional. In contrast with the aforementioned methodologies, the main advantage of this formulation is that neither the waveform of the solution nor the number of switching angles need to be a priori determined. On the contrary, the optimal waveform and switching angles are automatically obtained as a result of the minimization process.
\end{itemize}

In this contribution, we present an alternative optimal control approach to solve the SHM problem, based on the so-called adjoint methodology, which has been systematically associated to optimal control problems and their applications to several fields of science and engineering \cite{castro2007,giles2000}. A similar methodology has already been applied in our recent work \cite{biccari2021} to obtain multilevel controls for linear time-dependent control systems satisfying the Kalman rank condition. In this work, we will show how this dual technique can be adapted to solve the SHM problem.

This document is structured as follows. In Section \ref{sec:math_formulation}, we introduce the mathematical formulation of the SHM problem. In Section \ref{sec:OptimalControl}, we present the optimal control approach to SHM and its dual version, which is the main contribution of the present paper. Section \ref{sec:Numerics} is devoted to numerical examples of a concrete SHM problem that we have solved by means of our methodology. Finally, in Section \ref{sec:Conclusions}, we summarize the conclusions of our work and propose some open problems for future investigation.

\section{Preliminaries}\label{sec:math_formulation}

In this section, we present the mathematical formulation of the SHM problem and introduce the notation that will be used throughout the paper. Let 
\begin{align}\label{eq:Udef}
	\mathcal{U} = \{u_1, \ldots, u_L\}
\end{align}
be a given set of $L\geq 2$ real numbers satisfying
\begin{align*}
	& u_1 = -1, \; u_L = 1 \;\text{ and } \; u_{k+1} = u_k + \frac{2}{L-1}, \quad\forall k\in \{1,\ldots, L\}.
\end{align*}

The main objective of SHM is to construct a step function $u(t):[0,2\pi)\to\mathcal U$, with a finite number of switches, such that some lower-order coefficients in its Fourier expansion 
\begin{align*}
	u(t) = \sum_{j \in \mathbb{N}} \Big[a_j \cos(jt) + b_j \sin(jt)\Big]
\end{align*}
take specific values prescribed a priori.

Due to applications in power converters, it is typical to only consider functions with \textit{half-wave symmetry}, i.e. 
\begin{align}\label{eq:hwSymmetry}
	u(t + \pi) = -u(t)\quad \forall t \in [0,\pi).
\end{align}
Notice that, because of \eqref{eq:hwSymmetry}, the Fourier coefficients of $u$ corresponding to even indices $j$ vanish: 
\begin{align*}
	a_j &= \frac{1}{\pi} \int_0^{2\pi} u(\tau) \cos(j\tau)\,d\tau = \frac{1}{\pi} \int_0^{\pi} u(\tau) \cos(j\tau)d\tau + \frac{1}{\pi}\int_{\pi}^{2\pi} u(\tau)\cos(j\tau)\,d\tau
	\\
	&= \frac{1}{\pi} \int_0^{\pi} u(\tau)\Big[\cos(j\tau)-\cos(j(\tau+\pi))\Big]\,d\tau = \frac{1}{\pi} \Big(1-(-1)^j\Big)\int_0^{\pi} u(\tau)\cos(j\tau)\,d\tau,
	\\[15pt]
	b_j &= \frac{1}{\pi} \int_0^{2\pi} u(\tau) \sin(j\tau)\,d\tau = \frac{1}{\pi} \Big(1-(-1)^j\Big)\int_0^{\pi} u(\tau)\sin(j\tau)\,d\tau, 
\end{align*}
which gives the Fourier expansion
\begin{align}\label{eq:an}
	&u(t) = \sum_{\underset{j\, odd}{j \in \mathbb{N}}} \Big[a_j \cos(jt) + b_j \sin(jt)\Big] \notag
	\\
	&a_j = \frac{2}{\pi} \int_0^\pi u(\tau ) \cos(j \tau)\,d\tau,\quad b_j = \frac{2}{\pi} \int_0^\pi u(\tau)  \sin(j \tau)\,d\tau. 
\end{align}

Moreover, in view of \eqref{eq:hwSymmetry}, in what follows, we will only work with the restriction $u\vert_{[0,\pi)}$ which, with some abuse of notation, we still denote by $u$. Hence, from now on, we will consider piece-wise constant functions $u: [0,\pi)\to \mathcal{U}$ of the form 
\begin{align}\label{eq:uExpl}
	u (t)= \sum_{m=0}^M s_m\chi_{[\phi_m,\phi_{m+1})} (t), \quad M\in\mathbb{N}, 
\end{align}
where
\begin{itemize}
	\item $\mS = \{s_m\}_{m=0}^M$ with $s_m\in \mathcal{U}$ and $s_m\neq s_{m+1}$ for all $m\in \{0,\ldots, M\}$ is the waveform. 
	
	\vspace{0.1cm}
	\item $\Phi = \{ \phi_m\}_{m=1}^{M}$ are the switching angles such that
	\begin{align*}
		0= \phi_0 < \phi_1 <\ldots < \phi_M < \phi_{M+1} = \pi.
	\end{align*}
	
	\vspace{0.1cm}
	\item $\chi_{[\phi_m,\phi_{m+1})}$ denotes the characteristic function of the interval $[\phi_m,\phi_{m+1})$.
\end{itemize}
Observe that $\mS$ and $\Phi$ fully characterize any function $u$ of the form \eqref{eq:uExpl}. 

In the engineering applications that motivated our study, due to technical limitations, it is preferable to employ signals taking consecutive values in $\mathcal{U}$. In the sequel, we will refer to this property of the waveform as the \emph{staircase property}, which can be rigorously formulated as follows.
\begin{definition}\label{def:staircaseDef}
We say that a signal $u$ of the form \eqref{eq:uExpl} fulfills the \emph{staircase property} if its waveform $\mS$ satisfies
\begin{align}\label{eq:staircase}
	\left]s_m^{min},s_{m}^{max}\right[ \cap \mathcal{U} = \emptyset, \quad \forall m\in \{ 0, \ldots, M-1 \},
\end{align}
where $s^{min}_m = \min\{s_m,s_{m+1}\}$ and $s^{max}_m = \max\{s_m,s_{m+1}\}$.
\end{definition}

Note that when $\mathcal{U} = \{-1,1\}$ (which is known in the SHM literature as the bi-level problem), this property is satisfied for any $u$ of the form \eqref{eq:uExpl}.

We can now conclude this section by formulating the SHM problem as follows.

\begin{problem}[SHM]\label{pb:SHEp}
Let $\mathcal{U}$ be given as in \eqref{eq:Udef}, and let $\mathcal{E}_a$ and $\mathcal{E}_b$ be finite sets of odd numbers of cardinality $\vert\mathcal{E}_a\vert = N_a$ and $\vert\mathcal{E}_b\vert = N_b$ respectively. For any two given vectors $\aT \in \mathbb{R}^{N_a}$ and $\bT \in \mathbb{R}^{N_b} $, we want to construct a function $u: [0,\pi)\to\mathcal{U}$ of the form \eqref{eq:uExpl}, satisfying \eqref{eq:staircase}, such that the vectors ${\bf a} \in \mathbb{R}^{N_a}$ and ${\bf b} \in \mathbb{R}^{N_b}$, defined as
\begin{align}\label{eq:vectorsAB}
	{\bf a} = \big(a_j \big)_{j\in \mathcal{E}_a} \qquad \text{and} \qquad
	{\bf b} = \big(b_j \big)_{j\in \mathcal{E}_b}
\end{align}
satisfy ${\bf a} = \aT$ and ${\bf b} = \bT$, where the coefficients $a_j$ and $b_j$ in \eqref{eq:vectorsAB} are given by \eqref{eq:an}.
\end{problem}  

\section{SHM as an optimal control problem}\label{sec:OptimalControl}

In \cite{biccari2021A}, the SHM problem has been formulated via optimal control. In this formulation, the Fourier coefficients of the signal $u(t)$ are identified with the terminal state of a controlled dynamical system of $N:=N_a+N_b$ components defined in the time-interval $[0,\pi)$. The control of the system is precisely the signal $u(t)$, defined as a function $[0,\pi)\to \mathcal{U}$, which has to steer the state from the origin to the desired values of the prescribed Fourier coefficients. This approach is based on the observation (see \cite[Section 4]{biccari2021A} for more detail) that the Fourier coefficients of the function $u(t)$ can be rewritten as the initial state of a dynamical system controlled by $u(t)$
\begin{align}\label{eq:CauchyReversed}
	\begin{cases}
		\dot{\bf x}(t) = {\bf C}(t)u(t),  & t \in [0,\pi)
		\\[5pt]
		{\bf x}(0) = [\aT; \bT]=:{\bf x}_0,
	\end{cases}
\end{align}
with
\begin{equation}\label{eq:Dynamics}
	{\bf C}(t) = \left[{\bf C}^a(t),{\bf C}^b(t)\right]^\top, 
\end{equation}
and ${\bf C}^a(t) \in \RR^{N_a} $ and ${\bf C}^b(t) \in\RR^{N_b}$ given by
\begin{align}\label{eq:Bab}
	& {\bf C}^a(t) = -\frac 2\pi\Big[\cos(e_a^1t),\cos(e_a^2t),\cdots,\cos(e_a^{N_a}t)\Big] \notag 
	\\
	\\
	& {\bf C}^b(t) = -\frac 2\pi\Big[\sin(e_b^1t),\sin(e_b^2t),\cdots,\sin(e_b^{N_b}t)\Big]. \notag 
\end{align}
Here, $e_a^i$ and $e_b^i$ denote the elements in $\mathcal{E}_a$ and $\mathcal{E}_b$, i.e.
\begin{align*}
	\mathcal{E}_a = \{e_a^1,e_a^2,e_a^3,\dots,e_a^{N_a}\}, \quad \mathcal{E}_b = \{e_b^1,e_b^2,e_b^3,\dots,e_b^{N_b}\}.
\end{align*}

We notice that the unique solution to \eqref{eq:CauchyReversed} can be characterized through a direct integration against a test function $\phi\in\RR^N$ as 
\begin{align}\label{eq:weakSol}
\langle {\bf x}(t) - {\bf x}_0,\phi\rangle = \int_0^t \langle u(\tau),{\bf C}^\top(\tau)\phi\rangle\,d\tau,
\end{align}
where $\langle\cdot,\cdot\rangle$ is the standard scalar product in $\RR^N$. We shall use this \eqref{eq:weakSol} in the proof of our main result Theorem \ref{thm:control}.

With this formulation, the SHM Problem \ref{pb:SHEp} can be reduced to a controllability one for \eqref{eq:CauchyReversed}.

\begin{problem}[SHM via controllability]\label{pb:SHEpControl}
Let $\mathcal{U}$ be given as in \eqref{eq:Udef}, and let $\mathcal{E}_a$, $\mathcal{E}_b$ and the targets $\aT$ and $\bT$ be given as in Problem \ref{pb:SHEp}. We look for a function $u: [0,\pi)\to [-1,1]$ of the form \eqref{eq:uExpl}, satisfying \eqref{eq:staircase}, such that the solution ${\bf x} \in C([0,\pi);\RR^N)$ to \eqref{eq:CauchyReversed} satisfies ${\bf x}(\pi) = 0$.
\end{problem}

In \cite{biccari2021A}, we have shown that Problem \ref{pb:SHEpControl} can be solved via an optimal control approach, by computing 
the optimal SHM control $\ue \in L^1([0,\pi);[-1,1])$ through
\begin{align}\label{eq:directFunct}
	&\ue = \underset{{\underset{{\bf x} \text{ solves } \eqref{eq:CauchyReversed}}{u \in L^1([0,\pi);[-1,1])}}}{\text{argmin}} \Fe(u)
	\\
	&\Fe(u)=\frac{1}{2\varepsilon} \|{\bf x}(\pi)\|^2_{\RR^N} + \int_0^\pi \mathcal{L}(u(t))\,dt,\notag
\end{align}
where $\varepsilon>0$ is a (small) penalization parameter and the function $\mathcal L$ is constructed as the piece-wise affine interpolation of the parabola $\mathcal P(u) = u^2$ in $[-1,1]$, considering the elements in $\mathcal{U}$ as the interpolating points:
\begin{align}\label{eq:PLP}
	&\mathcal{L}(u) = \begin{cases}
		\lambda_k(u) & \text{if }  u \in [u_k,u_{k+1}) \\ 1 & \text{if } u = u_{L} 
		\end{cases} 
		\\
		&\notag \text{for all } u\in[-1,1] \text{ and } k \in \{1,\dots,L-1\}, 
\end{align}
where 
\begin{align}\label{eq:lambda_k}
	\lambda_k(u):= (u_{k+1}+u_k)u - u_ku_{k+1}.
\end{align}
In particular, we have proved in \cite[Theorems 4.1-4.2 and Proposition 4.1]{biccari2021A} that: 
\begin{itemize}
	\item The optimal control $\ue$, solution to Problem \ref{pb:SHEpControl} is unique and has the form \eqref{eq:uExpl} satisfying \eqref{eq:staircase}.
	\item $\ue$ is continuous with respect to the initial datum ${\bf x}_0$ in the strong topology of $L^1(0,\pi)$.
	\item The optimal trajectory $\xe$ associated with $\ue$ satisfies
	\begin{align}\label{eq:approxControl}
		\|\xe(\pi)\|^2_{\RR^N} \leq 4\varepsilon\pi\|\mathcal{L}\|_\infty,
	\end{align}
	that is, $\ue$ is an \textit{approximate control} solution of the SHM problem, allowing us to get $\varepsilon$-close to the target Fourier coefficients. 	
\end{itemize}

As we anticipated, the main interest of this paper is to show that the SHM Problem \ref{pb:SHEpControl} can be solved through duality. As a matter of fact, applying general results of the Fenchel-Rockafellar theory \cite{ekeland1999}, we can build the following dual problem associated with \eqref{eq:directFunct}
\begin{align}\label{eq:adjointFunPrel}
	&\pe = \underset{{\bf p}_\pi\in\RR^N}{\text{argmin}} \;\Je({\bf p}_\pi)
	\\
	&\Je({\bf p}_\pi)=\int_0^\pi \mathcal L^\star\big({\bf C}^\top(t){\bf p}(t)\big)\,dt + \frac{\varepsilon}{2}\norm{{\bf p}_\pi}{\RR^N}^2 + \langle {\bf x}_0,{\bf p}(0)\rangle, \notag 
\end{align}
where ${\bf p}$ is the solution of the adjoint equation
\begin{align}\label{eq:adjointEq}
	\begin{cases}
		\dot{\bf p}(t) = 0, & t\in [0,\pi)
		\\
		{\bf p}(\pi) = {\bf p}_\pi\in\RR^N
	\end{cases}
\end{align}
and $\mathcal L^\star$ denotes the convex conjugate of $\mathcal L$ (see \cite[Chapter 3]{boyd2004} for the definition of $\mathcal L^\star$). 

Moreover, we notice that in the particular situation of the SHM problem considered in this work, the adjoint dynamics is trivial and the solution of \eqref{eq:adjointEq} is constant, ${\bf p}(t) = {\bf p}_\pi$ for all $t\in [0,\pi)$. Hence, \eqref{eq:adjointFunPrel} simplifies into
\begin{align}\label{eq:adjointFun}
	&\pe = \underset{{\bf p}_\pi\in\RR^N}{\text{argmin}}\; \Je({\bf p}_\pi)
	\\
	&\Je({\bf p}_\pi)=\int_0^\pi \mathcal L^\star\big({\bf C}^\top(t){\bf p}_\pi\big)\,dt + \frac{\varepsilon}{2}\norm{{\bf p}_\pi}{\RR^N}^2 + \langle {\bf x}_0,{\bf p}_\pi\rangle, \notag 
\end{align}
where the dynamical behavior is given only by the time-dependent matrix ${\bf C}(t)$. 

Let us stress that the convex conjugate of a piece-wise affine function is still piece-wise affine and convex. In particular, for $\mathcal L$ defined as in \eqref{eq:PLP}-\eqref{eq:lambda_k}, $\mathcal L^\star$ can be computed explicitly and takes the form (see \cite[Lemma A4]{biccari2021}, \cite[Chapter 3]{boyd2004} and \cite[Section 3.31]{boyd2004A}) 
\begin{equation}\label{eq:PLPconjugate}
	\begin{array}{ll}
		\mathcal{L}^\star(\omega) = \begin{cases}
			\lambda_1^\star(\omega) & \text{if }  \omega \in [\omega_0,\omega_1) 
			\\
			\lambda_\ell^\star(\omega) & \text{if } \omega \in\,[\omega_{\ell-1},\omega_\ell)
			\\
			\lambda_L^\star(\omega) & \text{if } \omega \in [\omega_\ell,\omega_L) 
		\end{cases}, 
		& \;\begin{array}{l} \forall \omega\in [\omega_0,\omega_L] \\[5pt] \ell \in \{2,\dots,L-1\} \end{array},
	\end{array}  
\end{equation}
where 
\begin{equation}\label{eq:points}
	\begin{array}{l}
		\displaystyle\omega_0:=\omega_1-\frac{4}{L-1}, \quad \omega_L:=\omega_{L-1}+\frac{4}{L-1}
		\\[10pt]
		\displaystyle\omega_\ell:= u_\ell+u_{\ell+1}, \quad\forall \ell \in \{2,\dots,L-1\}, 	
	\end{array} 	
\end{equation}
and
\begin{align}\label{eq:lambda_k_conjugate}
	\lambda_k^\star(\omega):= u_k \omega - u_k^2, \quad\forall k\in \{1,\dots,L\}.
\end{align}

Dual optimization problems as \eqref{eq:adjointFunPrel} are widely employed in the control literature as they usually reduce the computational effort with respect to their \textit{primal} counterpart \eqref{eq:directFunct} when actually addressing the optimization process. This is essentially because the optimization variable in \eqref{eq:directFunct} is the control $u\in L^1([0,\pi);[-1,1])$, which is time-dependent, while in \eqref{eq:adjointFunPrel} one minimizes over ${\bf p}_\pi\in\RR^N$, the final state of the adjoint equation \eqref{eq:adjointEq}. This renders an optimization process which is, typically, computationally less expensive. We refer to \cite{boyer2013} for a more complete discussion on these issues in the context of the heat equation.

In what follows, we will show that the multilevel staircase control solution to the SHM Problem \ref{pb:SHEpControl} can be obtained through the minimization of \eqref{eq:adjointFun} and is given by
\begin{align}\label{eq:control}
	\ue(t)\in \partial\mathcal L^\star\big({\bf C}^\top(t)\pe\big), \quad\quad\text{ for a.e. } t\in [0,\pi),
\end{align}
where $\partial\mathcal L^\star$ denotes the subdifferential of the (non-differential) function $\mathcal L^\star$, defined for any ${\bf q}\in\RR^N$ as
\begin{align*}
	\partial \mathcal L^\star({\bf q}) = \Big\{\mathcal C\in \RR : \mathcal L^\star({\bf r}) - \mathcal L^\star({\bf q}) \geq \mathcal C({\bf r}-{\bf q}),\;\forall {\bf r}\in\RR^N\Big\}. 
\end{align*}

Besides, since $\partial\mathcal L^\star$ coincides with the standard derivative wherever the function $\mathcal L^\star$ is differentiable, we  have from \eqref{eq:control} that $\ue$ is given by the slopes of the linear branches $\lambda_k^\star$ in \eqref{eq:lambda_k_conjugate}, except for the $t_m$ such that ${\bf C}^\top(t_m)\pe=\omega_m$, $\omega_m$ being one of the values given in \eqref{eq:points}. This, in particular, will imply that the controls computed through \eqref{eq:adjointFun} have the desired multilevel and staircase structure. As a matter of fact, we have the following result.
\begin{theorem}\label{thm:control}
For any $\varepsilon > 0$, there exists a unique minimizer $\pe\in\RR^N$ of the functional $\Je$. Moreover, this minimizer is related with the minimizer $\ue$ of $\Fe$ through the formulas
\begin{align}\label{eq:ueChar}
	\ue(t)\in \partial\mathcal L^\star\big({\bf C}^\top(t)\pe\big), \quad\quad\text{ for a.e. } t\in [0,\pi)
\end{align}
and 
\begin{align}\label{eq:xeChar}
	\xe(\pi) = -\varepsilon\pe.
\end{align}
In particular, the approximate controllability condition \eqref{eq:approxControl} is satisfied, i.e. $\pe$ univocally determines a multilevel and staircase approximate control $\ue$, solution to the SHM Problem \ref{pb:SHEpControl}.
\end{theorem}

\begin{proof}
First of all, the existence and uniqueness of a minimizer $\pe\in\RR^N$ for $\Je$ is an immediate consequence of the direct method of calculus of variations, since the functional is clearly continuous, strictly convex and coercive, i.e.
\begin{align*}
	\Je({\bf p}_\pi)\to +\infty \quad\text{ as }\quad \norm{{\bf p}_\pi}{\RR^N}\to+\infty.
\end{align*} 

Let us show that, if $\ue$ is in the form \eqref{eq:ueChar}, then \eqref{eq:xeChar} holds. To this end, we first notice that the minimizer $\pe$ can be characterized by the Euler-Lagrange equation
\begin{align*}
	\frac{d}{d\sigma}\Je(\pe + \sigma\tilde{\bf p}_\pi)\Big|_{\sigma = 0} = 0, \quad\quad\forall\;\tilde{\bf p}_\pi\in\RR^N,
\end{align*}
which, in our case, reads as
\begin{align}\label{eq:E-Ladjoint}
	-\langle \varepsilon\pe + {\bf x}_0,\tilde{\bf p}_\pi\rangle \in \int_0^\pi \big\langle\partial\mathcal L^\star\big({\bf C}^\top(t)\pe\big), {\bf C}^\top(t)\tilde{\bf p}_\pi\big\rangle\,dt.
\end{align}
In the same way, the minimizer $\ue$ of \eqref{eq:directFunct} can be characterized by the Euler-Lagrange equation
\begin{align}\label{eq:E-Ldirect}
	0\in \int_0^\pi \langle\partial\mathcal L(\ue(t)),v(t)\rangle\,dt + \frac 1\varepsilon\langle \xe(\pi),\tilde{\bf x}(\pi)\rangle, \quad \forall \; v\in L^1([0,\pi);[-1,1]),
\end{align}
where $\tilde{\bf x}$ is the solution of \eqref{eq:CauchyReversed} corresponding to $\tilde u$. Moreover, let us recall that, by using \eqref{eq:weakSol} with $\phi = \tilde{\bf p}_\pi$, the optimal trajectory $\xe$ is characterized as
\begin{align*}
	\langle \xe(\pi) - {\bf x}_0,\tilde{\bf p}_\pi\rangle \in \int_0^\pi \big\langle\partial\mathcal L^\star\big({\bf C}^\top(t)\pe\big),{\bf C}^\top(t)\tilde{\bf p}_\pi\big\rangle\,dt.
\end{align*}
Combining this with \eqref{eq:E-Ladjoint}, we get that 
\begin{align*}
	\langle \xe(\pi) - {\bf x}_0,\tilde{\bf p}_\pi\rangle = -\langle \varepsilon\pe + {\bf x}_0,\tilde{\bf p}_\pi\rangle,
\end{align*}
from which we immediately obtain that $\xe(\pi) = -\varepsilon \pe$.

Let now $v(t)\in L^1([0,\pi);[-1,1])$ and let ${\bf z}_v(t)$ be the solution of 
\begin{align}\label{eq:mainEqZ}
	\begin{cases}
		\dot{\bf z}_v(t) = {\bf C}(t)v(t), & t\in [0,\pi)
		\\
		{\bf z}_v(0) = 0
	\end{cases}
\end{align}
Multiplying \eqref{eq:mainEqZ} by $\pe$ and integrating by parts using \eqref{eq:xeChar} we obtain that 
\begin{align*}
	0 &= \int_0^\pi \langle \dot{\bf z}_v(t)-{\bf C}(t)v(t),\pe\rangle\,dt = \langle {\bf z}_v(\pi),\pe\rangle - \int_0^\pi \langle v(t),{\bf C}^\top(t)\pe\rangle\,dt 
	\\
	&= -\frac 1\varepsilon \langle {\bf z}_v(\pi),\xe(\pi)\rangle - \int_0^\pi \langle v(t),{\bf C}^\top(t)\pe\rangle\,dt. 
\end{align*}
Combining this with \eqref{eq:E-Ldirect}, we then have that 
\begin{align*}
	0\in \int_0^\pi \langle \partial\mathcal L(\ue(t))-{\bf C}^\top(t)\pe,v(t)\rangle \,dt.
\end{align*}
Since the above relation holds for all $v(t)\in L^1([0,\pi);[-1,1])$, we then conclude that
\begin{align*}
	{\bf C}^\top(t)\pe \in \partial\mathcal L(\ue(t)), \quad\quad\text{ for a.e. } t\in[0,\pi),
\end{align*}
which, thanks to \cite[Lemma A2]{biccari2021}, is equivalent to 
\begin{align*}
	\ue(t)\in \partial\mathcal L^\star({\bf C}^\top(t)\pe), \quad\quad\text{ for a.e. } t\in[0,\pi).
\end{align*}

Finally, the multilevel and staircase structure of $\ue$ in \eqref{eq:ueChar} is a direct consequence of the fact that, by construction, $\mathcal L^\star$ is a piece-wise affine function and, therefore $\partial\mathcal L^\star$ is piece-wise constant. The complete details of this last part of the proof are omitted here for brevity, and can be found in \cite{biccari2021}.
\end{proof}

\begin{remark}
According to Theorem \ref{thm:control}, our dual methodology provides a solution to the SHM Problem \ref{pb:SHEp} under the perspective of approximate controllability, meaning that we can construct a multilevel and staircase approximate control $\ue$ allowing us to get $\varepsilon$-close to the target Fourier coefficients. For completeness, we shall stress that in the literature (see, e.g., \cite{fabre1995}) approximate controllability had also been addressed through a slightly different approach, by minimizing the functional
\begin{align}\label{eq:adjointFunExact}
	&\Jeh({\bf p}_\pi) = \int_0^\pi \mathcal L^\star\big({\bf C}^\top(t){\bf p}_\pi\big)dt + \varepsilon \norm{{\bf p}_\pi}{\RR^N} - \langle {\bf x}_0,{\bf p}_\pi\rangle.
\end{align}

In this case, the coercivity of $\Jeh$, hence the existence of an optimal control $\pe$ is guaranteed by the unique continuation principle 
\begin{align}\label{eq:unique}
	{\bf C}(t)^\top {\bf p}(t) = 0 \;\text{ for a.e }t\in[0,\pi) \;\Rightarrow\; {\bf p}_\pi = 0,
\end{align}
which is trivially satisfied by the adjoint dynamics \eqref{eq:adjointEq}. This is the Fenchel-Rockafellar dual of the problem
\begin{align}\label{eq:directFunctExact}
	&\ue = \underset{\underset{\underset{\norm{x(\pi)}{\RR^N}\leq \varepsilon}{{\bf x} \text{ solves } \eqref{eq:CauchyReversed}}}{u \in L^1([0,\pi);[-1,1])}}{\text{argmin}} \int_0^\pi \mathcal{L}(u(t))\,dt
\end{align}

Nevertheless, as it has been observed in \cite{glowinski1994,boyer2013}, to solve \eqref{eq:adjointFunExact} or \eqref{eq:directFunctExact} numerically may lead to important computational difficulties, which can instead be avoided through \eqref{eq:adjointFun}

\end{remark} 

\section{Numerical experiments}\label{sec:Numerics}

In this section, we present some numerical experiments to illustrate how our dual optimal control problem \eqref{eq:adjointFun} allows to solve the SHM Problem \ref{pb:SHEpControl}. To this end, we consider the frequencies 
\begin{align}\label{eq:frequencies}
	\mathcal{E}_a = \mathcal{E}_b = \{1,5,7,11,13\}
\end{align}
and the target vectors 
\begin{align}\label{eq:targets}
	\aT = \bT = (m,0,0,0,0)^\top\in\RR^{5}, \; \forall m \in [-0.8,0.8].
\end{align}
In the engineering applications of SHM motivating our work, $m$ in \eqref{eq:targets} is the so-called \textit{modulation index}. Finally, we choose $\varepsilon=10^{-5}$ and the control set 
\begin{align}\label{eq:levels}
	\mathcal{U} = \left\{-1,0,1\right\},
\end{align}
meaning that $u^\ast$ will be a $3$-level staircase control function taking only the constant values in \eqref{eq:levels}. 

As for the minimization of $\Je$, we adopt the approach of \cite[Sections 6.1 and 6.2]{biccari2021A} and implement the interior point method via the optimization software \texttt{IPOpt} \cite{wachter2006implementation}, with the help of the open-source tool \texttt{CasADi} \cite{andersson2019casadi} for nonlinear optimization and algorithmic differentiation. 

Our experiments were conducted on a personal Acer laptop with 1.6 GHz Intel Core i5-10210U processor and 16GB RAM. The results of our simulations are displayed in Fig. \ref{fig:control} and \ref{fig:controlSide}, in which we have plotted the function
\begin{align*} 
\begin{array}{cccc}
	\Phi: & [-0.8, 0.8]\times [0,\pi] & \longrightarrow & \mathcal{U} 
	\\
	& (m,t)  &\longmapsto & u_m^\ast (t),
\end{array}
\end{align*}
where, for each $m \in [-0.8,0.8]$, $u_m^\ast$ represents the solution to the SHM problem in our case of study. That is, Fig. \ref{fig:control} and \ref{fig:controlSide} contain all the solutions of the SHM problem in the range of modulation indices $m \in [-0.8,0.8]$ and with the frequencies, targets coefficients and levels defined in \eqref{eq:frequencies}, \eqref{eq:targets} and \eqref{eq:levels}, respectively. 

In Fig. \ref{fig:control}, $u_m^\ast$ is displayed seen from above. Each vertical snapshot in the picture corresponds to a multilevel control for a specific modulation index, which moves from one level to another each time there is a change of color. Moreover, to better appreciate the multilevel and staircase structure of the controls obtained, we present in Fig. \ref{fig:controlSide} a side view of $u_m^\ast$. In both figures, we can appreciate how our numerical results are in accordance with Theorem \ref{thm:control}.

\begin{figure}[h]
	\centering
	\includegraphics[scale=0.48]{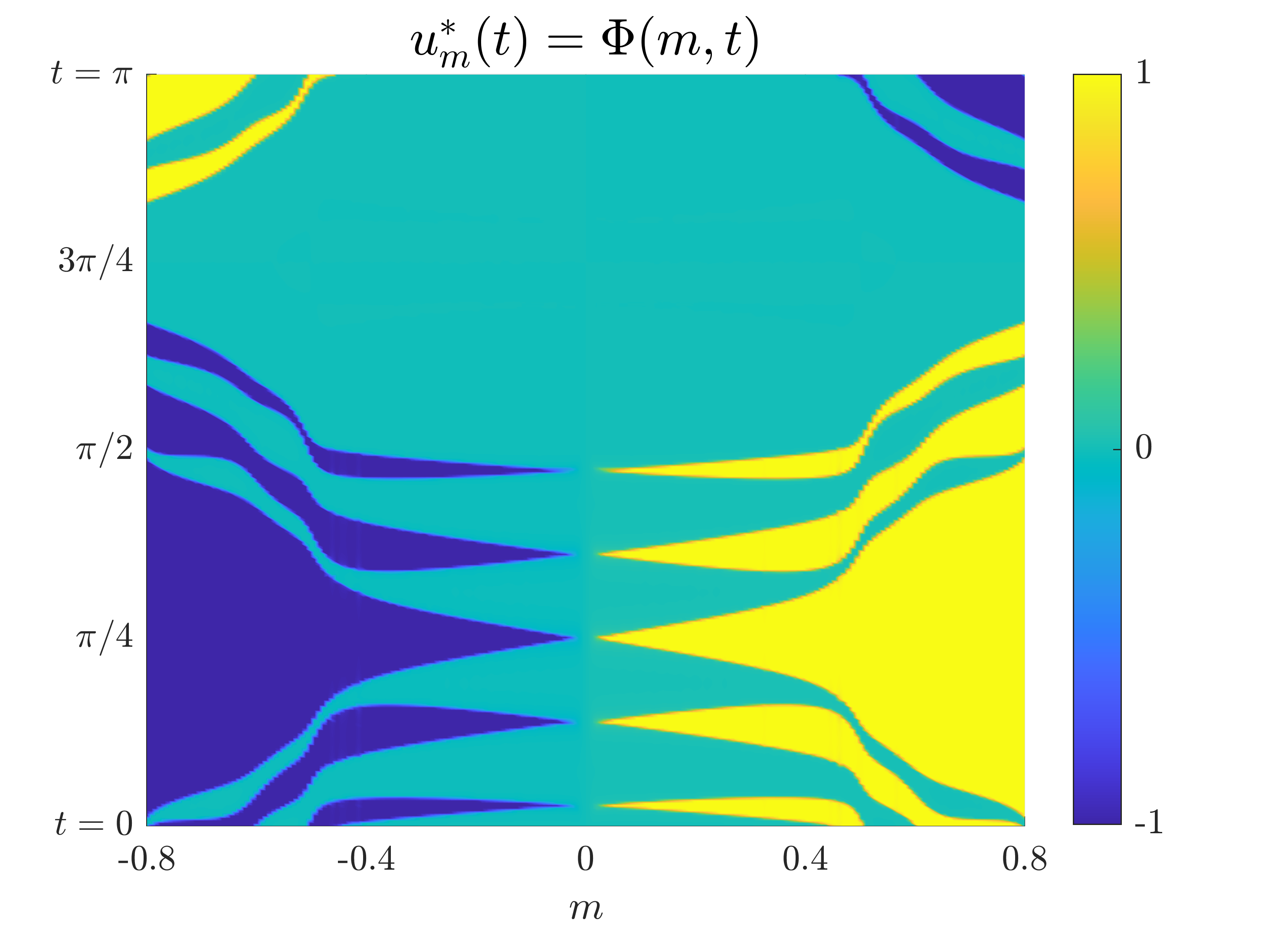}
	\caption{Top view of the $9$-levels staircase control for the SHM Problem \ref{pb:SHEpControl} computed via the minimization of $\Je$.}\label{fig:control}
\end{figure} 

\begin{figure}[h]
	\centering
	\includegraphics[scale=0.48]{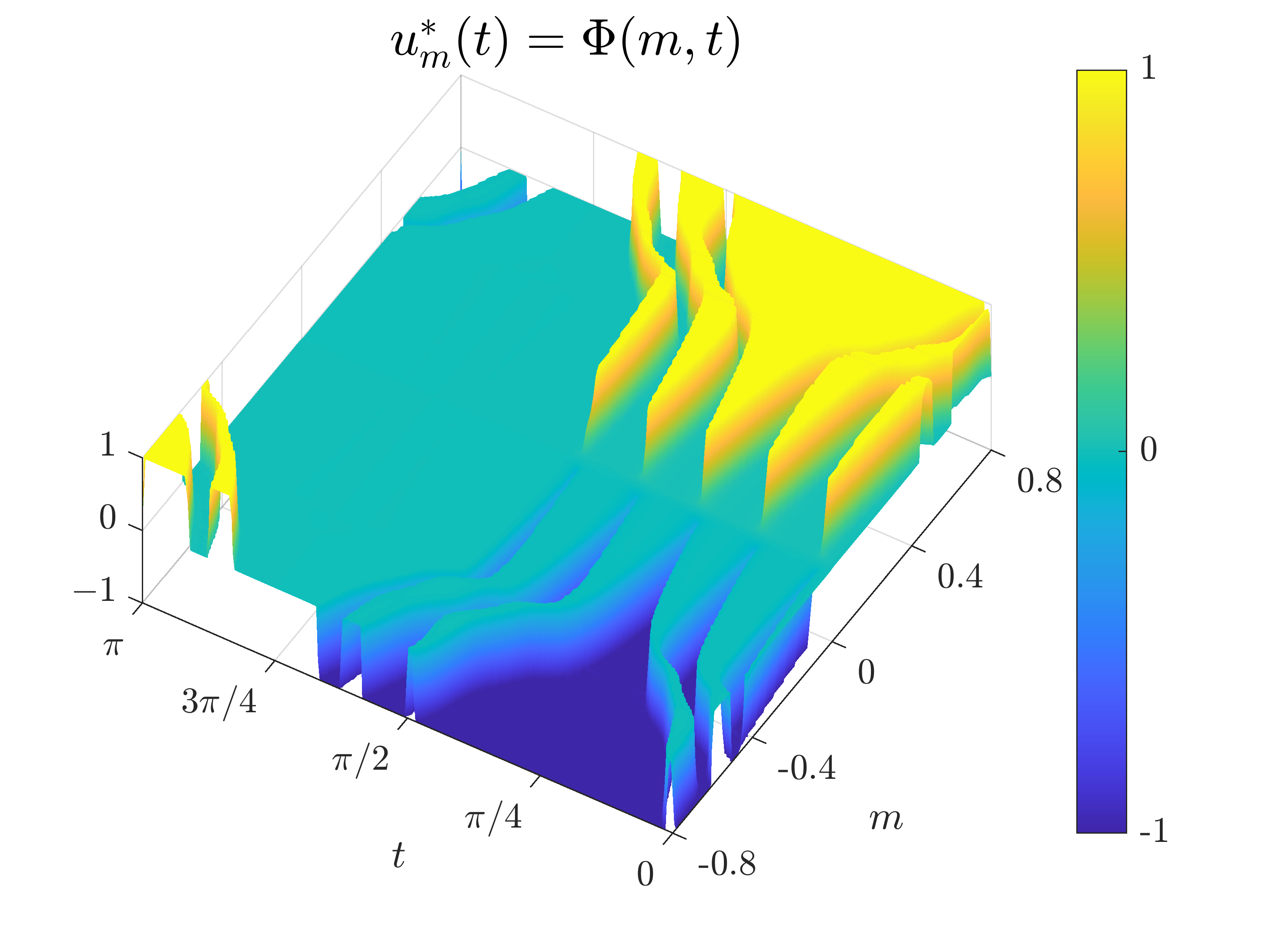}
	\caption{Side view of the $9$-levels staircase control for the SHM Problem \ref{pb:SHEpControl} computed via the minimization of $\Je$.}\label{fig:controlSide}
\end{figure} 

\section{Conclusions}\label{sec:Conclusions}

In this paper, we have discussed the Selective Harmonic Modulation problem in power electronic engineering adopting an optimal control viewpoint. More precisely, we have illustrated how the SHM Problem \ref{pb:SHEp} can be interpreted as an approximate controllability one for which the SHM signal $u$ plays the role of the control and can be obtained through a dual minimization process. Besides, we have shown both theoretically and through numerical simulations that, by suitably designing the penalization $\mathcal L^\star$ in our cost functional as the convex conjugate of the piece-wise affine approximation of a parabola, the proposed dual optimal control problem automatically generates SHM signals having the desired multilevel and staircase structure. 

As already observed in our previous contribution \cite{biccari2021A}, this optimal control approach to SHM solves several critical issues  arising in practical power electronic engineering applications. In particular, neither the waveform nor the number of switching angles need to be a priori determined, as they are implicitly established by the optimal control. This bypasses the (usually difficult) task of a priori estimating the number of switches which is necessary to reach the desired Fourier coefficients. 

However, some relevant issues are not covered by our study, and will be considered in future works:

\begin{itemize}
	\item[1.] \textbf{SHM via null controllability}. The results of this paper show that the SHM Problem \ref{pb:SHEp} can be addressed as an approximate controllability one. It would then be natural to analyze whether this can be extended to the null controllability framework, which corresponds to taking the limit $\varepsilon\to 0^+$ in \eqref{eq:adjointFun}. This would require some precise estimate of the computational cost with respect to $\varepsilon$.
	\item[2.] \textbf{Characterization of the solvable set}. The SHM Problem \ref{pb:SHEp}, by its own nature, brings some restrictions when converted into a controllability one. Firstly, the time horizon for the control problem is $\pi$, therefore fixed and bounded. Secondly, the energy of the controls is also bounded, as $u$ can only take the values in the control set $\mathcal U$. These restrictions yield that not all the initial data ${\bf x}_0\in\RR^N$ in \eqref{eq:CauchyReversed} are controllable, meaning that the SHM problem may not be solvable. It would then be interesting to have a full characterization of the solvable set for the SHM problem, thus determining the entire range of Fourier coefficients which can be reached by means of our approach.	A rough estimate of the solvable set can be obtained by considering \eqref{eq:adjointFun} with $\varepsilon=0$ and noticing that, as shown in \cite[Section 3]{biccari2021}, there exists a constant $\mathcal C>0$ such that 
	\begin{align}\label{eq:observability}
		\int_0^\pi \mathcal L^\star({\bf C}(t)^\top{\bf p}_\pi)\,dt \geq \mathcal C\norm{{\bf p}_\pi}{\RR^N}.
	\end{align}
	The observability inequality \eqref{eq:observability} is equivalent to the unique continuation \eqref{eq:unique} and, therefore, yields that, in the limit $\varepsilon\to 0^+$, the coercivity of $\Je$ is guaranteed for small enough initial data, namely, $\norm{{\bf x}_0}{\RR^N}\leq \mathcal C$. On the other hand, for large initial data, this coercivity may be lost. We stress that this is just a sufficient condition showing that the SHM problem can be solved through null controllability if the solvable set is contained in some ball in the euclidean space $\RR^N$. It would be interesting to obtain a sharper characterization of this solvable set and, possibly, its geometry. 
	
	\item[3.] \textbf{Minimal number of switching angles}. In practical applications, to optimize the converters' performance, it is required to keep the number of switches in the SHM signal the lowest possible. This yields to the very relevant yet challenging task of designing a cost functional for our optimal control problem which provides a multilevel and staircase signal $u$, with the additional structural property of having the minimal possible number of switching angles. A similar problem has been treated, in a different context, in \cite{lin2014}. There, the authors addressed optimal discrete-valued control problems using the so-called \textit{control parametrization enhancing technique} and introducing in their cost functional a penalization term involving the $TV$-norm of the control. This term has precisely the effect of reducing the variations (hence, the switching) in the control signal. It would be worth to analyze whether a similar approach is applicable also in our context, for instance by penalizing the $TV$-norm of the control in \eqref{eq:directFunct}, and if it effectively yields to a minimal switches' number.
\end{itemize}

\bibliographystyle{abbrv}
\bibliography{biblio}             
                                                   







\end{document}